\documentclass{elsarticle}
\usepackage{amsmath,amssymb,amsfonts,amsthm}

\usepackage{tikz,tikz-3dplot}

\usepackage{enumerate,verbatim}
\usepackage{xcolor}
\usepackage{hyperref}
\definecolor{myBlue}{rgb}{0.0,0.0,0.55}
\hypersetup{
	plainpages=false,
    colorlinks,
    linktocpage=true,   
    linkcolor={red!50!black},
    citecolor={blue!50!black},
    urlcolor={blue!80!black}
} 

\usepackage{comment,multicol,xspace}

\newcounter{mnote}
  \let\oldmarginpar\marginpar
    \renewcommand\marginpar[1]{\-\oldmarginpar[\raggedleft\footnotesize #1]%
    {\raggedright\footnotesize #1}}

\newtheorem{theorem}{Theorem}[section]
\newtheorem{lemma}[theorem]{Lemma}
\newtheorem{corollary}[theorem]{Corollary}
\newtheorem{definition}[theorem]{Definition}

\newtheorem{assumption}[theorem]{Assumption}

\makeatletter
\newcommand{\tnorm}{\@ifstar\@tnorms\@tnorm}
\newcommand{\@tnorms}[1]{%
\left|\mkern-2.5mu\left|\mkern-2.5mu\left|
#1
\right|\mkern-2.5mu\right|\mkern-2.5mu\right|
}
\newcommand{\@tnorm}[2][]{%
\mathopen{#1|\mkern-2.5mu#1|\mkern-2.5mu#1|}
#2
\mathclose{#1|\mkern-2.5mu#1|\mkern-2.5mu#1|}
}
\makeatother

\numberwithin{theorem}{section}
\numberwithin{equation}{section}

\begin{document}

\begin{frontmatter}

\title{A note on the error estimate of the virtual element methods}

\author[UCIaddress]{Shuhao Cao\corref{mycorrespondingauthor}}
\cortext[mycorrespondingauthor]{Corresponding author}
\ead{scao@math.uci.edu}

\author[UCIaddress]{Long Chen}
\ead{chenlong@math.uci.edu}

\author[UCIaddress]{Frank Lin}
\ead{fmlin@uci.edu}

\address[UCIaddress]{Department of Mathematics, University of California Irvine, 
Irvine, CA 92697}

\begin{abstract}
This short note reports a new derivation of the optimal order of the a priori error 
estimates for conforming virtual element methods (VEM) on 3D polyhedral meshes based 
on an error equation. The geometric assumptions, which are necessary for the optimal 
order of the conforming VEM error estimate in the $H^1$-seminorm, are relaxed for 
that in a bilinear form-induced energy norm. 
\end{abstract}


\end{frontmatter}


\section{Introduction}
\label{sec:intro}
Virtual element method (VEM) (e.g., 
\cite{Beirao-da-Veiga;Brezzi;Cangiani;Manzini:2013principles,Beirao-da-Veiga;Brezzi;Marini;Russo:2014hitchhikers,Beirao-da-Veiga;Dassi;Russo:2017High-order})
 can be viewed as a 
universalization of the finite element method (FEM) on simplicial and cubical 
elements to any polytopal elements, enabling a much greater flexibility in mesh 
generations. All the relevant integral quantities (e.g., 
stiffness matrix) can be computed or approximated from the degrees of freedom in the 
VEM space without explicitly constructing the non-polynomial basis functions.

The aim of this paper is to present the optimal order of error estimates of VEM with 
relaxed geometric assumptions on a three dimensional mesh. Consider the following 
weak formulation for a toy model Poisson 
equation with zero Dirichlet boundary condition in a three dimensional polyhedral 
domain $\Omega$: given an $f\in L^2(\Omega)$, find $u\in H_0^1(\Omega)$ such 
that
\begin{equation}\label{eq:weakform}
a(u,v) :=(\nabla u, \nabla v) = (f, v)\quad \forall v\in H_0^1(\Omega),
\end{equation}
where $(\cdot,\cdot)$ is the inner product on $L^2(\Omega)$.

Then $a(\cdot,\cdot)$ is approximated by the following bilinear form 
$a_h(\cdot,\cdot)$ on the VEM space $V_h$ \eqref{eq:space-global} built upon a 
polyhedral partition $\mathcal{T}_h$ of $\Omega$. $a_h(\cdot,\cdot)$ is the 
summation of local bilinear forms on each element that contains an orthogonal 
$H^1$-projection term $\big(\nabla \Pi_K (\cdot),\nabla \Pi_K (\cdot)\big)$ and a 
stabilization term $S_K(\cdot, \cdot)$ to ensure the 
coercivity. The projection term can be computed exactly 
through the local degrees of freedom for functions of $u,v\in V_h$:
\begin{equation}
a_h(u,v)=\;\sum_{K\in \mathcal{T}_h} 
\Bigl[\big(\nabla \Pi_K u,\nabla \Pi_K v\big)_K+S_K(u,v)\Bigr],
\end{equation}
where $\Pi_K$ is the $H^1$-projection operator (see Definition \ref{def:proj-H1}) to 
the space of degree $k$ polynomials on $K$.

In the traditional VEM approaches, the stabilization term is $S_K(\cdot, \cdot)$ 
constructed to satisfy $k$-consistency, as well as the following norm equivalence 
that holds between the exact $H^1$-inner product $a(\cdot,\cdot)$ and the 
approximated form $a_h(\cdot,\cdot)$ on the VEM space,
\begin{equation}\label{eq:normequiv}
a(u,u) \lesssim a_h(u,u)\lesssim a(u,u), \quad \text{for }\; u\in V_h,
\end{equation}
in which both constants in the inequalities are independent of $u$, but are 
dependent of the geometry of the meshes. With this 
property the finite dimensional approximation problem to the weak formulation 
(\eqref{eq:VEM 
problem} in Section \ref{sec:pre}) using the VEM discretization 
is well-posed, and it can be proved that the error estimate under the $H^1$-seminorm 
is optimal (e.g., see 
\cite{Beirao-da-Veiga;Brezzi;Cangiani;Manzini:2013principles}). One 
possible choice of the local stabilization on $K$ is given in 
\cite{Beirao-da-Veiga;Brezzi;Cangiani;Manzini:2013principles}: 
\begin{equation}
\label{eq:stab-orig}
S^{\text{orig}}_K(u,v)=\sum_{r=1}^{N_K} \chi_r(u-\Pi_K u)\chi_r(v-\Pi_K v),
\quad \text{ for }\;u,v\in V_h,
\end{equation} 
where $N_K$ is the number of degrees of freedom on an element $K$ (see Definition 
\ref{DOF}), $\chi_r$ is each individual degree of freedom.

Under certain geometric assumptions of the polytopal mesh, the aforementioned norm 
equivalence \eqref{eq:normequiv} is established with a 
proper choice of the stabilization (e.g., \eqref{eq:stab-orig}). 
Typical geometric assumptions include that (1) the mesh is star-shaped with 
chunkiness parameter uniformly bounded below (uniform star shape; see Definition 
\ref{def:star-shape}), and (2) the 
distance 
between neighboring vertices are comparable to the diameter of the element (no short 
edges or small faces). However, it has been observed 
in numerical experiments that the optimal rates of convergence for the virtual 
element methods can be achieved extraordinarily with relatively little to none 
geometric constraints on meshes (see 
e.g.,~\cite{Berrone;Borio:2017Orthogonal,Benedetto;Berrone;Pieraccini;Scialo:2014virtual,Beirao-da-Veiga;Dassi;Russo:2017High-order,Chen;Wei;Wen:2017interface-fitted}.)

In \cite{Beirao-da-Veiga;Lovadina;Russo:2017Stability}, different choices of 
stabilization terms are analyzed in detail, and the geometric 
assumptions are further relaxed for two dimensional mesh by 
including the case where elements contain short edges. Recently, in \cite{Brenner;Sung:2018Virtual}, it is shown 
that one can allow small faces on a three dimensional mesh and still achieve 
the optimal order with only the assumption that the elements are uniform star shape. 
However, several error estimates still require faces to be uniform star shape and 
some error estimates depend on the logarithm of the longest to the shortest edge 
ratio of the faces. This rules out anisotropic faces with unbounded aspect ratio.

To partially explain the robustness of VEM with respect to shapes, we shall use a 
different approach, which was first proposed in 
\cite{Cao;Chen:2018Anisotropic} to handle the 2D cut mesh and in 
\cite{Cao;Chen:2018AnisotropicNC} for nonconforming VEM, to relax the geometric 
assumptions further on three dimensional meshes, and still achieve the optimal 
order error estimates. Instead of working on the stronger $H^1$-seminorm, the error 
analysis is performed toward a weaker ``energy norm'' 
$\tnorm{\cdot}:=a_h^{1/2}(\cdot,\cdot)$. 
Similar to that of the Discontinuous Galerkin (DG)-type methods, an error equation 
for $\tnorm{\cdot}$, is derived. 
This error equation breaks down the $\tnorm{u_h-u_I}$ 
into several standard projection and interpolation error estimates. Our 
method does not rely on the norm equivalence property of the stabilization term. 
Instead, different from the above so-called identity matrix stabilization 
\eqref{eq:stab-orig}, the stabilization term is concocted from the ``boundary 
term'' emerged from the integration by parts (see Section \ref{sec:erroreq} for 
detail), while equipped with correct weights to remain the optimal order for the 
error estimates.

The following new 
stabilization term is proposed in this paper, which is partly inspired by the 
conjecture in \cite{Brenner;Sung:2018Virtual} and is ``singularly conforming'' in 
the sense that the term which keeps the conformity of the method may have a small 
constant in front it.
\begin{equation}
\begin{aligned}
S_K(u,v)= &\;h_K^{-1}\;\sum_{F\subset \partial K}\Bigl[
\big(Q_K u-Q_F u,Q_K v-Q_F v\big)_{F}+
\\
& \epsilon_F h_F\sum_{e\subset \partial F} \big( u-Q_F u, v-Q_F v\big)_e\Bigr],
\end{aligned}
\end{equation}
where $Q_K$, $Q_F$ are $L^2$-projection operators (see Definition \ref{def:proj-L2}) 
to the local spaces of degree $k$ polynomials on $K$ and $F$, respectively. The 
$\epsilon_F$ 
is related to the chunkiness parameter of $\rho_F$ (see Definition 
\ref{def:star-shape}) of each polygonal face $F$ 
on the boundary of a polyhedral element $K$. 

This new approach, comparing to the 
existing results in 
\cite{Beirao-da-Veiga;Lovadina;Russo:2017Stability,Brenner;Sung:2018Virtual}, allows 
us to deal with the 
mesh that has less constraints on the shape regularity. For example, the chunkiness 
parameter $\rho_F$ 
of each face $F$ on an element $K$ may no longer be uniformly bounded below, i.e., 
the constants in the new estimates do not depend on the logarithm of the 
longest and the shortest edge of each 
face. As a result, we obtain the optimal order error 
estimate on a weaker energy norm \eqref{Triple estimate} with a set of relaxed 
geometric assumptions \ref{assumptions:geometry} that are 
introduced in Section \ref{sec:estimation}. 

The rest of the paper is organized as follows. We begin with a brief review of VEM 
definitions. Then we go through some error estimate lemmas from the past references 
(e.g., \cite{Brenner;Sung:2018Virtual}). For each inequality with constant we will 
put extra emphasis on whether the hidden constant depends on the chunkiness parameter
of the domain or not. After that, our main result, the error equation will be 
derived, 
and then based on this error equation, the optimal order of the a priori error 
estimate under appropriate geometric assumptions can be achieved.

\section{Preliminaries}
\label{sec:pre}

In this section, we introduce the definition of the VEM space, its modified variant, 
and the corresponding degrees of freedom. 
Throughout the paper, without explicitly define them, we will use standard notations 
for differential operators, function spaces and norms that can be found, for 
examples in \cite{adams2003sobolev}.

The domain $\Omega$ is partitioned into a three dimensional mesh $\mathcal{T}_h$, 
and for simplicity $\Omega$ is assumed to have a polyhedral boundary so that there 
is no geometric error of $\mathcal{T}_h$ on $\partial \Omega$. Let $K$ be a simple 
polyhedral element in $\mathcal{T}_h$. $F$ denotes 
a face of the element, and $e$ denotes an edge of a face. 
$D$ denotes a general domain in two or three dimensions, and $h_D$ is the 
diameter of $D$.

\subsection{VEM spaces}
To define the three dimensional VEM space, first we need to define the two 
dimensional local VEM space $V_k$ 
(\cite{Beirao-da-Veiga;Brezzi;Cangiani;Manzini:2013principles}) and the 
modified space $W_k$ (\cite{Ahmad;Alsaedi;Brezzi;Marini:2013Equivalent}). Notice 
when defining the local VEM space on a face, the surface Laplacian operator 
$\Delta_F$ on a face $F$ shall be used. Let $k\in \mathbb{N}$ and let 
$\mathbb{P}_k(D)$ be the space of polynomial functions with degree up to $k$ (where 
$\mathbb{P}_{-1}$ contains only zero polynomial.) on $D$.
\begin{definition} [Local two dimensional VEM space on a face $F$]
\label{def:space-vem-2d}
\begin{equation}
V_k(F):= \big\{v\in H^1(F): \Delta_{F} v \in \mathbb{P}_{k-2}(F), 
v|_{\partial F}\in B_k(\partial F) \big\},
\end{equation}
where
\begin{equation}
B_k(\partial F) := \{ v\in C^0(\partial F): v|_e\in \mathbb{P}_k(e) 
\text{ for all } e\subset \partial F \}.
\end{equation}
\end{definition}

The degrees of freedom of the space in Definition \ref{def:space-vem-2d} can be 
defined using 
the scaled monomials.
Let $D$ be a two dimensional simple polygon or three dimensional simple polygonal 
domain, 
and $(x_c, y_c, z_c)$ be the center of mass
of $D$. Then the scaled monomials are 
polynomials of the form 
$m_{\alpha}=(\frac{x-x_c}{h_D})^{\alpha_1}(\frac{y-y_c}{h_D})^{\alpha_2}(\frac{z-z_c}{h_D})^{\alpha_3}$
 where 
$\alpha_1$, $\alpha_2$, $\alpha_3$ are non-negative integers. We define the degree 
to be 
$\alpha=\alpha_1+\alpha_2+\alpha_3$.

\begin{definition} [Degrees of freedom on a face]
\label{DOF}
The degrees of freedom of $v_h$ in $V_k(F)$ are defined as follows:
\begin{enumerate}
\item The value of $v_h$ at the vertices of $F$.
\item The moments up to order $k-2$ of $v_h$ in each edge $e$. That is, 
$\displaystyle \frac{1}{|e|}\int_e{v_h}m_{\alpha}$ where $m_{\alpha}$ is a scaled 
monomial for $\alpha\leq k-2$.
\item the moments up to order $k-2$ of $v_h$ in $F$. That is, $\displaystyle 
\frac{1}{|F|}\int_F{v_h}m_{\alpha}$ where $m_{\alpha}$ is a scaled monomial for 
$\alpha\leq k-2$.
\end{enumerate}
\end{definition}



The following projection operator $\Pi_K$ in 
the gradient inner product can be defined for $H^1(D)$ functions in 2D or 3D, but 
only can be computed with the degrees of freedom above in 2D.

\begin{definition} [Gradient orthogonal projection operator]
\label{def:proj-H1}
$\Pi_D^k : H^1(D)\rightarrow \mathbb{P}_k(D)$, $v\mapsto \Pi_D^kv$ satisfies 
\[
\big(\nabla (\Pi_D^k v-v), \nabla p\big)_{D}=0, \quad \forall p\in 
\mathbb{P}_k(D).
\]
where the constant kernel is determined by the following constraint:
\begin{equation}\int_{D}(\Pi_D^k v-v)=0,\quad k\geq 2,
\end{equation} and
\begin{equation}\int_{\partial D}(\Pi_D^k v-v)=0,\quad k=1.
\end{equation}

\end{definition}

On a polygonal domain $D$, to compute the gradient projection of $v_h\in V_k (D)$
to $\mathbb{P}_k(D)$, it is 
sufficient to compute $(\nabla v_h, \nabla q)_D$ for all $q\in \mathbb{P}_k(D)$. 
Integration by parts yields
\begin{equation}
\label{eq:dof-int}
(\nabla v_h, \nabla  q)_D=-(v_h, \Delta q)_D+(v_h, \nabla q\cdot n)_{\partial D},
\end{equation}
then the first term of the right hand side can be computed via internal moments of  
$v_h$ in $D$ (See definition \ref{DOF}), and the second 
term can be computed because it is a polynomial integral (See definition 
\ref{def:space-vem-2d}). 

However, for a three dimensional polyhedron $D$, a naive generalization of 
the degrees of freedom for the local space $V_k(D)$ mimicking what of the polygonal 
version in \ref{DOF} is not sensible. In the three dimensional case, part of the 
second term $(v_h, \nabla q\cdot n)_{F}$ in \eqref{eq:dof-int} is a surface 
moment integral on $F$ that is not computable if $F$ is not triangular. The reason 
is that only the moments of $v_h$ 
on a face $F\subset \partial D$ up to degree $k-2$ are given as degrees of freedom 
(See definition \ref{DOF}), yet for $q\in 
\mathbb{P}_k(D)$, $\nabla q\cdot n|_F \in \mathbb{P}_{k-1}(F)$. To compute this, we 
need to be able to compute the $L^2$-projection onto $\mathbb{P}_{k-1}(F)$ for a VEM 
function $v_h$. To this end, modified face spaces such as $W_k(F)$ or 
$\Tilde{W}_k(F)$ are to be introduced.

\begin{definition} [$L^2$ orthogonal projection operator]
\label{def:proj-L2}
$Q_D^k: L^2(D)\rightarrow \mathbb{P}_k(D)$, $v\mapsto Q_D^kv $ satisfies
\[
\big(Q_D^k v-v, q\big)_{D}=0, \quad \forall q\in \mathbb{P}_k(D).
\]
\end{definition}

When $D$ is a polygonal face $F$ on the boundary of a polyhedron $K$, the above 
$L^2$ projection is not computable through the internal moment degrees of freedom 
for $V_k(F)$ in \ref{DOF}, in that the moments $(v_h, q)_D$ for polynomial $q$ being 
degree $k$ or $k-1$ are unknown. 
However the space $V_k(F)$ defined above can be enriched in a certain way 
(\cite{Ahmad;Alsaedi;Brezzi;Marini:2013Equivalent, Brenner;Sung:2018Virtual}, see 
definition \ref{WK} and 
\ref{QK}) such that the $L^2$-projection is computable from the same set of degrees 
of freedom with \ref{DOF}. These are the motivations behind defining the modified 
space such as $W_k(F)$ and $\Tilde{W}_k(F)$, instead of using a direct 
generalization from $V_k(F)$ to $V_k(D)$ for a polyhedron $D$.

When the order of the projection operators are omitted, we assume it is $k$, the 
same as the order of the VEM space.

\begin{definition} [Local modified VEM space 
\cite{Ahmad;Alsaedi;Brezzi;Marini:2013Equivalent}]
\label{WK}
Let 
$\widetilde{\mathbb{P}}_k(e)$ be 
the space of degree exactly $k$ monomials, then the local modified VEM space can 
be defined as:
\begin{equation}
\label{eq:space-WK}
\begin{aligned}
W_k(F):= \big\{v\in H^1(F): \;\; & \Delta_F v \in \mathbb{P}_{k}(F), 
v|_{\partial F}\in B_k(\partial F), 
\\
& (v,q)_{F}=(\Pi_F^k v,q)_{F}, \; \forall q\in \widetilde{\mathbb{P}}_k(F)\cup 
\widetilde{\mathbb{P}}_{k-1}(F)\big\}.
\end{aligned}
\end{equation}
\end{definition}

Note that $W_k$ and $V_k$ share the same degrees of freedom 
(\cite{Ahmad;Alsaedi;Brezzi;Marini:2013Equivalent}), but the $L^2$-projection of a 
function in $W_k$ is now computable.  In $W_k$ we 
can replace $(v_h, q)_K$ by $(\Pi_K^k v_h, q)_K$ for $q$ being degree $k$ or $k-1$ 
and the latter integral is computable.

The three dimensional local VEM space can be defined as follows:
\begin{definition} [Local three dimensional VEM space on an element $K$ 
\cite{Ahmad;Alsaedi;Brezzi;Marini:2013Equivalent}]
\label{3DVEM}
\begin{equation}
V_k(K):= \big\{v\in H^1(K): \Delta v \in \mathbb{P}_{k-2}(K), 
v|_{\partial K} \in B_k(\partial K) \big\},
\end{equation}
where $B_k(\partial K) := \big\{v\in C^0(\partial K): v|_F\in W_k(F), v|_e \in 
\mathbb{P}_{k}(e) \big\} $.
\end{definition}

Any function in $V_k(K)$ can be uniquely determined by its 
degrees of freedom (\cite{Ahmad;Alsaedi;Brezzi;Marini:2013Equivalent})
 defined in the following paragraph.

\begin{definition} [Degrees of freedom of the three dimensional VEM space]
\label{DOF_3d}We can take the following degrees of freedom of $v_h$ in $V_k(K)$, 
where $K$ is a three dimensional element.
\begin{enumerate}
\item The value of $v_h$ at the vertices of $K$
\item The moments on each edge $e$ up to degree $k-2$. That is, 
$\frac{1}{|e|}\int_e{v_h}m_{\alpha}$ where $m_{\alpha}$ is the scaled monomials 
with $\alpha\leq k-2$.
\item The moments on each face $F$ up to degree $k-2$. That is, 
$\frac{1}{|F|}\int_F{v_h}m_{\alpha}$ where $m_{\alpha}$ is the scaled monomials with 
$\alpha\leq k-2$.
\item The moments on the element $K$ up to degree $k-2$. That is, 
$\frac{1}{|K|}\int_K{v_h}m_{\alpha}$ where $m_{\alpha}$ is the scaled monomials with 
$\alpha\leq k-2$.
\end{enumerate}
\end{definition}
\vspace{1cm}

An alternative definition of the modified VEM space \cite{Brenner;Sung:2018Virtual}, 
that 
allows us to compute both $H^1$ and $L^2$ projection from degrees of freedom is the 
following. We denote such a space $\Tilde{W}_k(D)$, where $D$ can be a polyhedron 
domain in any dimension. For convenience we shall define 
$\Tilde{W}_k(e)=\mathbb{P}_k(e)$ for 
$e$ being 1 dimensional edge and higher dimension spaces are defined recursively. 

\begin{definition} [The modified local $\Tilde{W}_k$ space 
\cite{Brenner;Sung:2018Virtual}]
\label{QK}
Let $D$ be a two or three dimensional polygon or polygonal domain, define the space 
$\Tilde{W}_k(D)$ by 
\begin{equation}
\label{eq:space-QK}
\begin{aligned}
\Tilde{W}_k(D):= \big\{v\in H^1(D): \;\; & \Delta_D v \in \mathbb{P}_{k}(D), 
v|_{\partial D}\in C^0(\partial D), 
\\
& v|_{F}\in \Tilde{W}_k(F)\; \forall F\subset \partial D, \Pi_D^k v_h-Q_D^k v_h\in \mathbb{P}_{k-2}(D) \big\}.
\end{aligned}
\end{equation}
\end{definition}

When computing $L^2$ projection in $\Tilde{W}_k$, we first write $Q_K^k v_h=\Pi_K^k 
v_h+w$, $w\in \mathbb P_{k-2}$, and the corresponding integrals can be computed 
using internal degrees up to $k-2$.

We shall henceforth use the only $\Tilde{W}_k$ to define the global VEM space for 
the rest of the paper. Let 
\begin{equation}\label{eq:space-global}
V_h = \Tilde{W}_h := \{v_h\in H^1_0(\Omega)\cap C^0(\overline{\Omega}): v_h\vert_K 
\in \Tilde{W}_k(K), \; \forall K\in \mathcal{T}_h \}
\end{equation}
be the global VEM space for the rest of the paper, so that $L^2$ projection is 
computable for any three dimensional element $K$.

We then have the following natural definition of the nodal interpolation.

\begin{definition} [The canonical interpolation]
\label{def:interpolation}
For any $u\in H^1(\Omega)$, on $K\in \mathcal{T}_h$, $u_I$ is the local 
interpolation on $K$ which is defined as a function in $\Tilde{W}_k(K)$ that has the 
same degrees of freedom \eqref{DOF_3d} with $u$. Globally, $u_I$ is defined as a 
function in $V_h = \Tilde{W}_h$ that has the same degrees of freedom as $u$.
\end{definition}

We use the same notation $u_I$ for both local and global interpolation, but under 
the proper context it should not be confused.

The following choice of stabilization term is motivated by the error equation that 
will be derived in section \ref{sec:erroreq}. On an element $K$, the stabilization 
term is defined as follows:
\begin{equation}
\label{eq:sk-def}
\begin{aligned}
S_K(u,v) :=\; &h_K^{-1}\;\sum_{F\subset \partial K}
\Bigg[(Q_K u-Q_F u,Q_K v-Q_F v)_{F}
\\
&+\epsilon_F 
h_F\sum_{e\subset \partial F} ( (u-Q_F u), (v-Q_F v))_e\Bigg],
\end{aligned}
\end{equation}
 where $\epsilon_F\propto \rho_F$ is a mesh dependent constant
and the discrete bilinear form is given by 
\begin{equation}
\label{eq:ah-def}
a_h(u,v)=\;\sum_{K\in \mathcal{T}_h} 
\Bigl[\big(\nabla \Pi_K u,\nabla \Pi_K v\big)_K+S_K(u,v)\Bigr].
\end{equation}

Then the VEM approximation problem is: to seek $u_h\in 
V_h$ \eqref{eq:space-global} 
\begin{equation}\label{eq:VEM problem}
a_h(u_h,v_h) = \sum_{K\in \mathcal{T}_h} (f, Q_K v_h)\quad \forall v_h\in 
V_h .
\end{equation}

\subsection{Approximation Theory on Star-Shaped Domains}

In this section, we shall review some existing results on VEM projection 
\eqref{Projection} and 
interpolation error estimates \eqref{theorem:interpolation}. 

\begin{definition} [Star-shaped polytope]
\label{def:star-shape}
Let $D$ be a simple polygon or polyhedron. We said $D$ is star-shaped with respect 
to a disc/ball $B$ if for every point $y\in D$, the convex hull of $\displaystyle 
\{y\}\cup B$ is contained in $D$. If $D$ is star-shaped with respect to a disc/ball 
with radius $\rho h_D$. We define the supremum of $\rho$ to be chunkiness 
parameter $\rho_D$.
\end{definition}

\begin{lemma} [Bramble-Hilbert estimates on star-shaped domain 
\cite{brenner2007mathematical}]
\label{BH lemma star}
Let D be a star-shaped domain, then we have the following estimates,
\begin{equation}
\inf_{q\in \mathbb{P}_l}|u-q|_{H^m(D)}\leq C(\rho_D) 
h_D^{l+1-m}|u|_{H^{l+1}(D)},\text{  } \forall u\in H^{l+1}(D), l=0,1,...,k, m\leq l
\end{equation}
where $C(\rho_D)$ is inverse proportional to $\rho_D$.
\end{lemma}

The following scaled trace inequalities are often used when we need to bound norm on 
boundary faces by norm on elements.

\begin{lemma} [Trace inequalities on a star-shaped domain 
\cite{brenner2007mathematical}]
\label{Trace}
Let D be a star-shaped domain, then
\begin{equation}
\|u\|^2_{L^2(\partial D)}\lesssim h_D^{-1}\|u\|^2_{L^2(D)}+h_D|u|^2_{H^1(D)}, 
\forall u\in H^1(D),
\end{equation}
where the constant in $\lesssim$ is inverse proportional to $\rho_D$.
\end{lemma}

By combining the Bramble-Hilbert estimates, and the stability of projection 
operators ($Q_D$ and $\Pi_D$) (see 
\cite{Beirao-da-Veiga;Lovadina;Russo:2017Stability,Brenner;Sung:2018Virtual}) in 
$L^2$, $H^1$, and 
$H^2$ norms, we can obtain the following projection error estimates.

\begin{theorem} [Projection error estimate]
\label{Projection}
Let D be a star-shaped domain. Let $\Pi$ be $\Pi_D$ or $Q_D$ then for $m,l,k\in 
\mathbb{N}$, $0\leq m\leq 2$, $\min(1,m)\leq l\leq k$, $u\in H^{l+1}(D)$ we have
\begin{equation}
\displaystyle \|u-\Pi u\|_{m,D}\leq [C(\rho_D) h_D]^{l+1-m} |u|_{l+1,D}
\end{equation} where $C(\rho_D)$ is inverse proportional to $\rho_D$.
\end{theorem}

The optimal order of interpolation operators are much harder to prove. In 
\cite{Brenner;Sung:2018Virtual}, an auxiliary semi-norm is constructed to prove the 
following interpolation estimates. We will list the result here and refer the 
reference for the detail (although the estimate of $|u-Q_D u_I|_{1,D}$ is not 
explicitly given in \cite{Brenner;Sung:2018Virtual}, the derivation follows from 
$H^1$ 
stability of $Q_D$, and the procedure of deriving the estimate of 
$|u-\Pi_D u_I|_{1,D}$ is almost identical). The interpolation estimates for three 
dimensional element require the uniform star shape condition to be of optimal order.

\begin{theorem} [Interpolation error estimate \cite{Brenner;Sung:2018Virtual}]
\label{theorem:interpolation}
Let D be a star-shaped domain. Let $u_I$ be the nodal interpolation of the function 
on the local VEM space defined in \ref{def:interpolation}. We have, for $1\leq l\leq 
k$, 
$\forall u\in H^{l+1}(D)$
\begin{equation}
|u-u_I|_{1,D}+|u-\Pi_D u_I|_{1,D}+|u-Q_D u_I|_{1,D}\lesssim [C(\rho_D) 
h_D]^{l}|u|_{l+1,D}
\end{equation}
\begin{equation}
|u-\Pi_D u_I|_{2,D}\lesssim [C(\rho_D) h_D]^{l-1}|u|_{l+1,D}
\end{equation}
\begin{equation}
\|u-u_I\|_{0,D}+\|u-Q_D u_I\|_{0,D}+\|u-\Pi_D u_I\|_{0,D}\lesssim 
[C(\rho_D) h_D]^{l+1}|u|_{l+1,D}
\end{equation}
The constants $C(\rho_D)$ is inverse proportional to $\rho_D$. 
\end{theorem}

\section{A priori error estimate}
\label{sec:erroreq}
In this section, we first verify that the discrete bilinear form induces a norm on 
$V_h$, then an error equation based on the discrete bilinear form is derived. The a 
priori error estimate can be then derived from this error 
equation. 

Recall that on an element $K$, the bilinear form and the stabilization term are 
defined in \eqref{eq:ah-def} and \eqref{eq:sk-def}, and the VEM approximation 
problem is \eqref{eq:VEM problem}. Then the seminorm 
$\tnorm{v}=a_h^{1/2}(v,v)$ induced by $a_h(\cdot,\cdot)$ is verified to be a norm on 
the VEM space with the boundary condition imposed in the following lemma.

\begin{lemma}\label{lemma:norm}
$\tnorm{\cdot}$ is a norm on $V_h\cap H^1_0(\Omega)$.
\end{lemma}
\begin{proof}

It suffices to verify that if $\tnorm{v_h} = 0$, then $v_h\equiv 0$. By definition,
when $a_h(v_h,v_h)=0$, we have $\Pi_K v_h=0$ on each $K$, $Q_K v_h=Q_F v_h$ on 
each $F\subset \partial K$, and $v_h=Q_F v_h$ on each edge $e$.

By the boundary condition of the space, $Q_F v_h=0$ for $F$ on $\partial \Omega$. 
Because $Q_K v_h=Q_F v_h$ on each $F\subset \partial K$, that makes $Q_K v_h=0$ for 
$K$ that contains at least a boundary face. For the same reasons, $Q_K v_h=0$ for 
$K'$ that shares a face with $K$, which is an element that contains at least a 
boundary face. Repeat this argument we have $Q_K v_h=Q_F v_h=0$ for each $K$.

In addition, on each edge $v_h=Q_F v_h=0$. That makes the degrees of freedom of 
$v_h$ on each $K$ equal $0$, which in turn implies $v_h=0$ by unisolvence of the VEM 
space \cite{Beirao-da-Veiga;Brezzi;Cangiani;Manzini:2013principles}, which completes 
the proof.
\end{proof}
\begin{lemma} [An identity of the approximated bilinear form]
For $u$ that is the solution to \eqref{eq:weakform}, $u_h$ that is the solution 
to \eqref{eq:VEM problem}, and any $v\in V_h$, the following identity holds, 
\label{lemma:a_h identity}
\begin{equation}
\begin{aligned}
a_h(u_h, v) = \sum_{K\in \mathcal{T}_h}(\nabla \Pi_K u, \nabla \Pi_K v)_K +  
\sum_{K\in \mathcal{T}_h} \langle \nabla( u - \Pi_K u)\cdot n, Q_K v - Q_F 
v \rangle_{\partial K}. 
\end{aligned}
\end{equation}
\end{lemma}

\begin{proof}

First we apply the integration by parts to $u,v,\Pi_K u, Q_Kv$, and use the 
definitions of $H^1$-projection $\Pi_K$ and $L^2$-projection $Q_K$ to get
\begin{align*}
-(\Delta u, Q_K v)_K &= (\nabla \Pi_K u, \nabla Q_K v)_K 
+ \langle \nabla u\cdot n, 
Q_K v \rangle_{\partial K},\\[2pt]
(\Delta \Pi_K u, Q_K v)_K &= -(\nabla \Pi_K u, \nabla Q_K v)_K - \langle \nabla 
\Pi_K u\cdot n, Q_K v \rangle_{\partial K},\\[2pt]
\text{ and }\;- (\Delta \Pi_K u, v)_K &= (\nabla \Pi_K u, \nabla v)_K + \langle 
\nabla \Pi_K u\cdot n, Q_F v \rangle_{\partial K}.
\end{align*}
Adding above equations together and notice that $(\Delta \Pi_K u, Q_K v)_K = (\Delta 
\Pi_K u, v)_K$. By the definition of $Q_K$, we get
\begin{align*}
-(\Delta u, Q_K v)_K =\; &(\nabla \Pi_K u, \nabla v)_K \\
& +  \langle (\nabla u -\nabla \Pi_K u)\cdot n, Q_K v \rangle_{\partial K} \\
& +  \langle \nabla \Pi_K u\cdot n, Q_F v \rangle_{\partial K}.
\end{align*}
By definition, the first term can be rewritten as 
$(\nabla \Pi_K u, \nabla \Pi_K v)_K$. 
By the continuity of $\nabla u\cdot n$ across the interelement faces, and the fact 
that $Q_Fv$ is single value on the face $F$, 
$\sum_{K\in \mathcal{T}_h} \langle \nabla u \cdot n, Q_F v \rangle_{\partial 
K} = 0$. As a result, recalling the VEM approximation problem in equation 
\eqref{eq:VEM problem}, we arrive at the following identity
\begin{align*}
a_h(u_h, v)= &-\sum_{K\in \mathcal{T}_h} (\Delta u, Q_K v)_K
\\
= & \sum_{K\in \mathcal{T}_h}(\nabla \Pi_K u, \nabla \Pi_K v)_K +  
\sum_{K\in \mathcal{T}_h} \langle \nabla( u - \Pi_K u)\cdot n, Q_K v - Q_F 
v \rangle_{\partial K}.
\end{align*}
\end{proof}

\begin{theorem}[An error equation]
\label{theorem:err-eq}
Under the same setting with Lemma \ref{lemma:a_h identity}, let $u_I$ be the VEM 
interpolation in \eqref{def:interpolation}, and using the stabilization in 
\eqref{eq:sk-def}, the following identity holds,
\begin{equation}
\label{eq:error}
\begin{aligned}
a_h\bigl(u_h - u_I, v_h\bigr) =& \;\sum_{K\in \mathcal{T}_h} \Bigg[ (\nabla 
\Pi_K(u-u_I), \nabla \Pi_K v_h)_K
\\
& +\;\sum_{F\subset \partial K}\Bigl( (\nabla (\Pi_Ku-u)\cdot n, Q_K v_h-Q_F v_h)_{F}
\\&-h_K^{-1}(Q_K u_I-Q_F u_I,Q_K v_h-Q_F v_h)_{F}
\\&-\epsilon_F h_F\sum_{e\subset \partial F} (u_I-Q_F u_I,v_h-Q_F v_h)_e\Bigr)\Bigg]
\\
\end{aligned}
\end{equation}
\end{theorem}
\begin{proof}
It follows directly from Lemma \ref{lemma:a_h identity} and Definition 
\ref{eq:sk-def}.
\end{proof}

\begin{corollary} [An \textit{a priori} error bound] For a constant 
$\epsilon_F\propto \rho_F$, the following a priori error 
estimate holds for $u_h$ and $u_I$ (defined in \ref{def:interpolation}) with a 
constant independent of the chunkiness parameter for $\rho_F$ of each face in the 
underlying mesh,
\label{corollary:apriori}
\begin{equation}
\label{eq:apriori}
\begin{aligned}
 \tnorm{u_h-u_I}^2 \lesssim &\;\sum_{K\in \mathcal{T}_h} \Bigg[\|\nabla 
\Pi_K(u-u_I)\|^2_{0,K}
+\;\sum_{F\subset \partial K}\Bigl(h_K\|\nabla (\Pi_Ku-u)\cdot n\|^2_{0,F}
\\
&+h_K^{-1}\|Q_K u_I-Q_F u_I\|^2_{0,F}
+\epsilon_F\sum_{e\subset \partial F} \|u_I-Q_F u_I\|^2_{0,e}\Bigr)\Bigg].
\\
&
\end{aligned}
\end{equation}
\end{corollary}

\begin{proof}
From the error equation, plug in $v_h=u_h-u_I$ and apply the Cauchy-Schwarz 
inequality, we have 

\begin{equation}
\begin{aligned}
\tnorm{u_h-u_I}^2 \lesssim &\;\sum_{K\in \mathcal{T}_h} \Bigl[\|(\nabla 
\Pi_K(u-u_I)\|_{0,K}\|\nabla\Pi_Kv_h\|_{0,K}
\\  &+
 \;\sum_{F\subset \partial K}\Bigl(h_K^{1/2}\|\nabla (\Pi_Ku-u)\cdot 
 n\|_{0,F}h_K^{-1/2}\|Q_Kv_h-Q_F v_h\|_{0,F}
\\ \;
&+h_K^{-1/2}\|Q_K u_I-Q_F u_I\|_{0,F}h_K^{-1/2}\|Q_Kv_h-Q_F v_h\|_{0,F}+
\\
&\sum_{e\subset \partial F} \epsilon_F\| (u_I-\Pi_F u_I)\|_{0,e}\epsilon_F\| 
(v_h-\Pi_F v_h)\|_{0,e}\Bigr)\Bigr].
\end{aligned}
\end{equation}

The second part of each term is clearly parts of $\tnorm{u_h-u_I}$ and therefore can 
be bounded by $\tnorm{u_h-u_I}$. After cancelling $\tnorm{u_h-u_I}$ we get the 
estimate.
\end{proof}

\section{Geometric conditions and error estimations}
\label{sec:estimation}

In this section, based on the a priori error estimate in Corollary 
\ref{corollary:apriori}, the energy norm estimate follows from estimating each term 
in \eqref{eq:apriori}. The necessary geometry conditions, which are motivated by 
\eqref{eq:apriori} to have optimal order of convergence, are proposed as follows.

\begin{assumption}[Geometric conditions]
\label{assumptions:geometry}
For each element $K\in \mathcal{T}_h$, the following three geometric conditions are 
met:
\begin{enumerate}
\item Number of faces in $K$ is uniformly bounded.
\item $K$ is star-shaped with the chunkiness parameter $\rho_K$ defined in 
\ref{def:star-shape} bounded below. 
\item For each $F\subset \partial K$, $F$ is star-shaped, but the chunkiness 
parameter $\rho_F$ may not be uniformly bounded below.
\end{enumerate}
\end{assumption}
An example of the polyhedral element satisfying Assumption 
\ref{assumptions:geometry} is shown in Figure \ref{fig:anisoface}, on which the 
$H^1$-seminorm error estimate will have an undesirable $|\ln \epsilon|$ factor 
related to $\rho_F$ (see 
\cite{Beirao-da-Veiga;Lovadina;Russo:2017Stability,Brenner;Sung:2018Virtual}). 
Subsequently, we will show that on a weaker energy norm $\tnorm{\cdot}$, the a 
priori error estimate is independent of $\rho_F$ and only dependent on $\rho_K$.

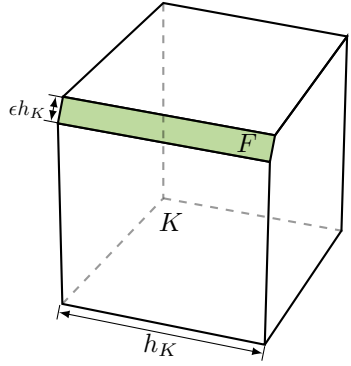
\begin{figure}[h]
\centering
\definecolor{ifemgreen}{RGB}{163,203,118}
\tdplotsetmaincoords{60}{110}
\begin{tikzpicture}[scale=1,point/.style = {draw, circle,  fill = black, inner sep 
= 1pt},>=latex, tdplot_main_coords]
\tikzset{every mark/.append style={solid,scale=2,opacity=0.8,fill=blue}}

\coordinate (A) at (0,-0.1,2.7);
\coordinate (B) at (-0.2,-0.1,3);

\coordinate (D) at (0,0,-0.2);
\coordinate (E) at (0.1,2.95,-0.2);

\draw[thick] (0.1,0.1,0) -- (0,0,2.7)-- (0,3,2.7) -- (0.2,3,0) -- cycle;

\draw[thick] (-0.2,0,3) -- (-3,0.4,3) -- (-3,3,3) -- (-0.2,3,3) -- cycle;

\draw[thick, fill=ifemgreen, fill opacity=0.7] (0,0,2.7)-- (0,3,2.7) 
--(-0.2,3,3) -- (-0.2,0,3) -- cycle;

\draw[thick] (0.2,3,0) -- (-3,2.92,0) -- (-3,3,3) -- (-0.2,3,3);

\draw[thick,dashed,opacity=0.4] (0.1,0.1,0) -- (-3,0.4,0) -- (-3,0.4,3) ;

\draw[thick, dashed,opacity=0.4] (-3,0.4,0) -- (-3,3,0);
    
\draw[|<->|] (A) -- (B);
\draw[|<->|] (D) -- (E);
\node[anchor = east,scale=0.8] at ($(A)!0.5!(B)$) {$\epsilon h_{K}$};
\node[anchor = south] at ($(D)!0.5!(E) + (0,0,-0.5)$) {$h_K$};

\node at (-0.2,2.6,2.8) {$F$};
\node at (0,1.6,1.5) {$K$};

\end{tikzpicture}
\caption{Set $\epsilon\to 0$ as $h_K\to 0$. $K$ is a cube without a prismatic slit 
with $\rho_K>1/2$ when $\epsilon$ is small, which satisfies all three assumptions in 
\ref{assumptions:geometry}. However, 
the chunkiness parameter $\rho_F\to 0$ for the marked face $F$, and this is 
problematic for the error estimate under $|\cdot|_1$. }
\label{fig:anisoface}
\end{figure}

\begin{lemma}[Optimal order error estimate of the stabilization term on a face] 
\label{lemma:stab-face}
Let $u\in H^{k+1}(K)$, for $k \geq 1$, and $u_I$ be the VEM space interpolation 
defined in \ref{def:interpolation}. Suppose the 
geometric assumptions \ref{assumptions:geometry} hold, then
\begin{equation}
h_K^{-1/2}\|Q_K u_I-Q_F u_I\|_{0,F}\lesssim h_K^{k}|u|_{k+1,K}.
\end{equation}
\end{lemma}

\begin{proof}
By triangle inequality,
\begin{equation}
\begin{aligned}
h_K^{-1/2}\|Q_K u_I-Q_F u_I\|_{0,F} & = h_K^{-1/2}\|Q_F(Q_K u_I- u_I)\|_{0,F}
\\ 
& \lesssim  h_K^{-1/2}\|Q_K u_I- u_I\|_{0,F}
\\
& \lesssim  h_K^{-1/2}(\|Q_K u_I-u\|_{0,F}+\|u-u_I\|_{0,F})
\\
& \lesssim  h_K^{-1}(\|Q_K u_I-u\|_{0,K}+\|u-u_I\|_{0,K})\\&+(|Q_K 
u_I-u|_{1,K}+|u-u_I|_{1,K})
\\
& \lesssim  h_K^{k}|u|_{k+1,K}
\end{aligned}
\end{equation}
where Theorems \ref{Projection}, \ref{theorem:interpolation}, \ref{Trace} are 
applied.
\end{proof}

\begin{lemma}[Optimal order error estimate of stabilization term on an edge] 
\label{stab order edge}
Under the same setting with Lemma \ref{lemma:stab-face}, for a mesh dependent 
constant $\epsilon_F\propto \rho_F$, we have
\begin{equation}
\epsilon_F \|u_I-Q_F u_I\|_{0,e}\lesssim \rho_F^{k-1}h_F^{k}|u|_{k+1,K},
\end{equation}
\end{lemma}

\begin{proof}

By the Theorem \ref{Trace} and triangle inequality, under the star-shaped 
condition \ref{def:star-shape},
\begin{equation}
\begin{aligned}
 \|u_I-Q_F u_I\|_{0,e} & \lesssim \epsilon_F^{-1} \big(h_F^{-1/2}\|u_I-Q_F 
 u_I\|_{0,F}+h_F^{1/2}|u_I-Q_F u_I|_{1,F}\big)
\\
& \lesssim \epsilon_F^{-1}\big( h_F^{-1/2} \|u_I-u\|_{0,F}+
h_F^{-1/2} \|u-Q_F u_I\|_{0,F}
\\
& \quad +h_F^{1/2}|u_I-u|_{1,F}+h_F^{1/2}|u-Q_F u_I|_{1,F}\big)
\end{aligned}
\end{equation}
where each except the last term has optimal error order by Theorem 
\ref{theorem:interpolation}. In order to use Theorem 
\ref{theorem:interpolation} on the face, $\rho_F$ need to be included because we do 
not assume it is uniformly bounded below, therefore the constant $\epsilon_F\propto 
\rho_F$ is introduced. For the $|\Pi_F u_I-Q_F u_I|_{1,F}$ term, we apply the 
inverse inequality on the face $F$ and the triangle inequality.

\begin{equation}
\begin{aligned}
 |u-Q_F u_I|_{1,F}  & \lesssim |u-\Pi_F u_I|_{1,F}+|\Pi_F u_I-Q_F u_I|_{1,F}
\\
& \lesssim |u-\Pi_F u_I|_{1,F}+\epsilon_F^{-1} h_F^{-1}\|\Pi_F u_I-Q_F u_I\|_{0,F}
\\
& \leq |u-\Pi_F u_I|_{1,F}+\epsilon_F^{-1} h_F^{-1}\|\Pi_F u_I- 
u_I\|_{0,F}\\
&\quad  + \epsilon_F^{-1} h_F^{-1}\|u_I-Q_F u_I\|_{0,F}
\end{aligned}
\end{equation}
where each term has optimal error order by Theorem \ref{theorem:interpolation}. 
Similarly 
the inverse inequality depends on $\rho_F$ 
\cite{Brenner;Sung:2018Virtual},
a mesh dependent constant $\epsilon_F\propto \rho_F$ is introduced to 
compensate.\end{proof}

Now we turn to derive the estimates of the other terms in the a priori error bound 
\eqref{corollary:apriori}.

\begin{lemma} [The projection type error estimates]
\label{lemma:projection}
Under the same setting with Lemma \ref{lemma:stab-face}, we have
\begin{equation}
\|\nabla \Pi_K(u-u_I)\|_{0,K}\lesssim h_K^{k}|u|_{k+1,K},
\end{equation}
and
\begin{equation}
h_K^{1/2}\|\nabla (u-\Pi_K u)\|_{0,F}\lesssim h_K^{k}|u|_{k+1,K}.
\end{equation}
\end{lemma}
\begin{proof}
By Theorem \ref{Projection} and $\Pi_K$ is the projection under 
$|\cdot|_{1,K}$,
$$ \|\nabla \Pi_K(u-u_I)\|_{0,K}\lesssim  |u-u_I|_{1,K}\lesssim h_K^{k}|u|_{k+1,K}.$$
In addition, by Theorems \ref{Trace} and \ref{Projection}
$$h_K^{1/2}\|\nabla (\Pi_Ku-u)\|_{0,F}\lesssim 
|\Pi_Ku-u|_{1,K}+h_K|\Pi_Ku-u|_{2,K}\lesssim h_K^{k}|u|_{k+1,K}.$$
\end{proof}

\begin{theorem}[Energy norm error estimate]
\label{Triple estimate}
If \eqref{eq:weakform} has a sufficiently regular solution $u \in H^{k+1}(K)$ 
($k\geq 1$), let $u_h$ be the VEM approximation in \eqref{eq:VEM problem}, and $u_I$ 
be the interpolation in Definition \ref{def:interpolation}, then under the geometric 
assumptions \ref{assumptions:geometry}, we have the following estimate 
\begin{equation}
\label{eq:est-energy}
\tnorm{u_h-u_I} \lesssim h^{k}|u|_{k+1} .
\end{equation}
\end{theorem}

\begin{proof}
\eqref{eq:est-energy} follows immediately from the bound of each term in the a 
priori error estimate \eqref{eq:apriori} by Lemmas \ref{lemma:projection}, 
\ref{lemma:stab-face}, and \ref{stab order edge}.
\end{proof}

\section*{Acknowledgement}
This work was supported in part by the National Science Foundation
under grant DMS-1418934.

\section*{References}


\end{document}